\documentclass{amsart}
\usepackage{mathtools}

\usepackage{enumitem} 
\newlist{enumroman}{enumerate}{1}
\setlist[enumroman]{font=\normalfont,label=(\roman*),leftmargin=0.3in}

\usepackage[]{hyperref}

\theoremstyle{plain}
\newtheorem{thm}{Theorem}[section]
\newtheorem{cor}[thm]{Corollary}
\newtheorem{lem}[thm]{Lemma}
\newtheorem{pro}[thm]{Proposition}
\theoremstyle{definition}
\newtheorem{defi}[thm]{Definition}
\newtheorem{rem}[thm]{Remark}
\newtheorem*{ackn}{Acknowledgment}
\title[Graph equivariant cohomological rigidity for GKM~graphs]{Graph equivariant cohomological rigidity\\for GKM~graphs}
\author{Matthias Franz}
\address{Department of Mathematics, University of Western Ontario, London, Ont.\ N6A\;5B7, Canada}
\email{mfranz@uwo.ca}
\author{Hitoshi Yamanaka}
\address{Osaka City University Advanced Mathematical Institute, 3-3-138, Sugimoto, Sumi\-yoshi-ku, Osaka, 558-8585 Japan}
\email{yamanaka@sci.osaka-cu.ac.jp}
\begin{document}

\begin{abstract}
We formulate the notion of an isomorphism of GKM~graphs. We then show that two GKM~graphs have isomorphic graph equivariant cohomology algebras if and only if 
the graphs are isomorphic.
\end{abstract}
\maketitle

\section{Introduction}

In~\cite[Thm.~7.2]{GKMequivariant}, Goresky--Kottwitz--MacPherson found
a remarkable combinatorial description of the equivariant cohomology, with complex coefficients, of a complex projective variety with an algebraic torus action
having only finitely many $1$-dimensional orbits.
Subsequently, Guillemin--Zara~\cite[\S1.6]{GZequivariant} generalized
their work 
to a wide class of closed $T$-manifolds, where $T$ is a compact torus. Nowadays, a manifold appearing in this class is called a GKM~manifold.

A GKM~manifold~$X$ determines an edge-labeled graph~$\mathcal{G}_{X}$ that encodes the structure
of the equivariant $1$-skeleton of~$X$ and the weights of the tangential real representations. Using these data one can define the graph equivariant cohomology~$H_{T}^{*}(\mathcal{G}_{X})$ in a purely combinatorial way.
It is a sub graded $H^{*}(BT)$-algebra of the algebra of functions from the fixed point set~$X^{T}$ to~$H^{*}(BT)$.

The theorem of Goresky--Kottwitz--MacPherson and of Guillemin--Zara mentioned above asserts that
if the equivariant cohomology of a GKM~manifold~$X$ is free over~$H^{*}(BT)$, then $H_{T}^{*}(X)$ is canonically
isomorphic to~$H_{T}^{*}(\mathcal{G}_{X})$ as an $H^{*}(BT)$-algebra.
This holds for complex coefficients, and also for integral coefficients provided that all isotropy groups in~$X$ are connected, see Remark~\ref{rem:different} below.

The work of Guillemin--Zara has another important aspect: they established an axiomatic formulation of the framework above by introducing
the notion of an abstract GKM~graph~$\mathcal{G}$ and its graph equivariant cohomology~$H_{T}^{*}(\mathcal{G})$ with integral coefficients, which is a graded algebra over the integral cohomology $H^{*}(BT)$.

Toric manifolds are important examples of GKM~manifolds.
In~\cite[Thm.~1.1]{Mequivariant} Masuda proved that the equivariant isomorphism type of a toric manifold,
considered as a complex algebraic variety with an algebraic torus action, is completely determined by its torus equivariant cohomology algebra with integral coefficients.
This work has led 
to a classification problem in toric topology which is nowadays called the cohomological rigidity problem. The aim of the present note is to generalize Masuda's result to arbitrary GKM~graphs.

Let $\mathcal{G}$ and~$\mathcal{G}'$ be two abstract GKM~graphs
defined for the same torus~$T$ (see Definition~\ref{defofGKM-graph}). We denote by~$H_{T}^{*}(\mathcal{G})$ and~$H_{T}^{*}(\mathcal{G}')$ the corresponding graph equivariant cohomology of~$\mathcal{G}$ and~$\mathcal{G}'$, respectively (see Definition~\ref{defofgraph}).
In Definition~\ref{defofisomorphism} we will introduce the notion of an isomorphism~$\varphi \colon \mathcal{G}'\rightarrow \mathcal{G}$. 
Our main theorem is the following:
\begin{thm}
  \label{mainthm}
$H_{T}^{*}(\mathcal{G})$ and~$H_{T}^{*}(\mathcal{G}')$ are isomorphic as $H^{*}(BT)$-algebras if and only if $\mathcal{G}$ and~$\mathcal{G}'$ are isomorphic as GKM~graphs.
\end{thm}

For a toric manifold~$X$, the GKM~graph~$\mathcal{G}_{X}$ and the fan~$\Sigma _{X}$ are essentially the same object.
Since isotropy groups in toric varieties are connected, Theorem~\ref{mainthm} generalizes Masuda's equivariant rigidity theorem to abstract GKM~graphs.

Throughout this note, we fix a compact torus~$T$ of rank~$r$ as well as positive integers~$n$ and~$n'$.
Note that $H^{*}(BT)$ can be regarded as the polynomial ring~$\mathbb{Z}[x_{1},\ldots,x_{r}]$ with the grading $\text{deg}\ x_{i}=2$.
In particular, it is a unique factorization domain.
For two polynomials~$P,Q\in H^{*}(BT)$, we write $P\mid Q$ if $Q=RP$ for some~$R\in H^{*}(BT)$.
We denote by~$|S|$ the number of elements of a finite set~$S$.

\section{Graph equivariant cohomology}
In this section we recall the notion of an abstract GKM~graph and its graph equivariant cohomology. The original paper is~\cite{GZequivariant}. We also introduce the notion of an isomorphism of GKM~graphs.

Let $\mathcal{G}$ be a finite $n$-valent undirected graph (multi-edges are allowed, but loops are not) with vertex set~$\mathcal{V}$.
We denote by~$\mathcal{E}$ the set of directed edges of~$\mathcal{G}$. (Note that $\mathcal{E}$ is not the set of edges of~$\mathcal{G}$; the cardinality of~$\mathcal{E}$ is twice that of the edge set.) For each~$e\in\mathcal{E}$, we denote by
$\overline{e}$ the directed edge obtained by reversing the direction of~$e$. Let $i(e)$ and~$t(e)$ be the initial and terminal point of a directed edge~$e$, 
respectively. We also use the following notation for vertices~$p$ and~$q$:
\[
\mathcal{E}_{p}\coloneqq \{\, e\in\mathcal{E}\mid i(e)=p\,\},
\]
\[
\mathcal{E}_{pq}\coloneqq \{\, e\in\mathcal{E}\mid i(e)=p,\:t(e)=q\,\}.
\]
\begin{defi}\label{defofaxial}
  A map~$\alpha \colon \mathcal{E}\rightarrow H^{2}(BT)$ is called an \textit{axial function} on~$\mathcal{G}$ if it satisfies the following three conditions
  for all~$e$,~$e'\in\mathcal{E}$:
\begin{enumroman}
\item \label{defofaxial-1}
  $\alpha (\overline{e})=\pm\alpha (e)$.
\item 
  (\textit{GKM~condition}) $\alpha (e)$ and~$\alpha (e')$ are linearly independent over~$\mathbb{Z}$ if $e\not= e'$ and~$i(e)=i(e')$.
\item \label{defofaxial-3}
  (\textit{Primitivity}) The greatest common divisor of the coefficients of~$\alpha (e)$ is $1$.
\end{enumroman}
\end{defi}

\begin{defi}\label{defofparallel}
Let $\alpha $ be an axial function on~$\mathcal{G}$. A \textit{parallel transport} of~$(\mathcal{G},\alpha )$ is a family~$\mathcal{P}=\{\mathcal{P}_{e}\}_{e\in\mathcal{E}}$ of 
bijections~$\mathcal{P}_{e}\colon \mathcal{E}_{i(e)}\rightarrow \mathcal{E}_{t(e)}$ satisfying the following conditions
for all~$e\in\mathcal{E}$ and all~$e'\in\mathcal{E}_{i(e)}$:
\begin{enumroman}
\item 
  $\mathcal{P}_{\overline{e}}=\mathcal{P}_{e}^{-1}$.
\item 
  $\mathcal{P}_{e}(e)=\overline{e}$.
\item \label{defofparallel-3}
  $\alpha (\mathcal{P}_{e}(e'))-\alpha (e')\in \mathbb{Z}\,\alpha (e)$.
\end{enumroman} 
\end{defi}

\begin{defi}\label{defofGKM-graph}
An \textit{abstract GKM~graph} (or simply \textit{GKM~graph}) of type~$(r,n)$ is a pair~$(\mathcal{G},\alpha )$ having at least one parallel transport.
\end{defi}

\begin{rem}
  \label{rem:different}
The above notation and terminology are somewhat different from the usual ones. Let us explain the differences.
\begin{enumroman}
\item
  Condition~\ref{defofaxial-1} in Definition~\ref{defofaxial} is weaker than the usual requirement~$\alpha (\overline{e})=-\alpha (e)$. Our definition is motivated by
  the notion of a torus graph introduced in~\cite[\S 3]{MMPtorus}, and is more natural from the point of view of real manifolds.
\item
  Condition~\ref{defofaxial-3} in Definition~\ref{defofaxial} is related to our choice of integers as the coefficient ring for graph equivariant cohomology.
  For complex coefficients, it would hold trivially.

  If we want the theorem of Goresky--Kottwitz--MacPherson and of Guille\-min--Zara to hold in the case of integral coefficients,
  we have to put further assumptions on the GKM~manifold~$X$.
  One possible condition is the connectedness of the stabilizer group~$T_{x}$ for any~$x\in X$, see~\cite[Thm.~1.1]{FPexactcohomology}
  or \cite[Thm.~2.1]{FPexactsequence}. The primitivity condition reflects the connectedness of the stabilizer groups for the equivariant $1$-skeleton in a purely algebraic fashion.
  Many known GKM~manifolds with effective torus action satisfy it.
\item
  The family~$\mathcal{P}$ is usually called a connection on~$(\mathcal{G},\alpha )$. As well-explained in~\cite{GZequivariant}, this terminology owes its origin to the ``fiber bundle" picture for GKM~graphs, see~\cite[\S 1.7]{GZequivariant}.
  However, it seems more appropriate to call it a parallel transport since each bijection~$\mathcal{P}_{e}$ corresponds to an identification of fibers.
\end{enumroman}
\end{rem}

\begin{defi}\label{defofgraph}
  The \textit{graph equivariant cohomology} 
  of a GKM~graph~$(\mathcal{G},\alpha )$ is defined to be
\begin{equation*}
  H_{T}^{*}(\mathcal{G}) = 
  \Bigl\{\, f\colon \mathcal{V}\rightarrow H^{*}(BT)\Bigm| \alpha (e)\mid (f(i(e))-f(t(e)))\ (e\in\mathcal{E})\,\Bigr\} \; ;
\end{equation*}
it is a sub graded $H^{*}(BT)$-algebra of
the algebra of all functions~$\mathcal{V}\to H^{*}(BT)$.
We denote by~$H^{2i}_{T}(\mathcal{G})$ its degree~$2i$ component where
 $f\in H_{T}^{*}(\mathcal{G})$ is of degree~$2i$ if $f(p)$ is so for any~$p\in \mathcal{V}$.
\end{defi}

We now introduce the notion of an isomorphism of GKM~graphs.
For~$p\ne q\in\mathcal{V}$, we set
\[
P_{pq}\coloneqq \prod_{e\in \mathcal{E}_{pq}}\alpha (e).
\]
Note that $P_{pq}\ne1$ if and only if $p$ and~$q$ are adjacent.
Let $(\mathcal{G}',\alpha ')$ be a GKM~graph of type~$(r,n')$.
\begin{defi}\label{defofisomorphism}
  An \textit{isomorphism}~$\varphi \colon \mathcal{G}'\rightarrow \mathcal{G}$ of GKM~graphs is a bijection $\varphi_{\mathcal{V}} \colon \mathcal{V}'\rightarrow \mathcal{V}$
  such that for all~$p'$,~$q'\in\mathcal{V'}$ one has
  $P_{\varphi _{\mathcal{V}}(p')\varphi _{\mathcal{V}}(q')}=\pm P_{p'q'}$.
  \par
  This implies that $p'$ and~$q'$ are adjacent if and only if $\varphi _{\mathcal{V}}(p')$ and~$\varphi _{\mathcal{V}}(q')$ are so.
  Two GKM graphs~$\mathcal{G}$,~$\mathcal{G}'$ are said to be \textit{isomorphic} if there exists an isomorphism~$\mathcal{G}'\rightarrow \mathcal{G}$ of GKM~graphs.
\end{defi}

In light of the primitivity condition,
the criterion stated in
Definition~\ref{defofisomorphism} can be paraphrased as follows:

For any two vertices~$p'$,~$q'$ of~$\mathcal{G}'$ there exists a bijection~$\varphi _{\mathcal{E}}\colon \mathcal{E}_{\varphi _{\mathcal{V}}(p')\varphi _{\mathcal{V}}(q')}\rightarrow \mathcal{E}_{p'q'}$
such that $\alpha' (\varphi _{\mathcal{E}}(e))=\pm\alpha (e)$.
Here we are using that each weight~$\alpha (e)$ is a prime element of the UFD $H^{*}(BT)$ by the primitivity condition.

\begin{rem} \label{iso}
The preceding reformulation implies that any isomorphism~$\varphi \colon \mathcal{G}'\rightarrow \mathcal{G}$ induces a graded $H^{*}(BT)$-algebra isomorphism $\varphi ^{*}\colon H_{T}^{*}(\mathcal{G})\rightarrow H_{T}^{*}(\mathcal{G}')$ defined by~$(\varphi ^{*}(f))(p')\coloneqq f(\varphi _{\mathcal{V}}(p'))$.
The assignment is functorial in the sense that $\text{id}_{\mathcal{G}}^{*}=\text{id}_{H_{T}^{*}(\mathcal{G})}$ and $(\psi \circ \varphi )^{*}
=\varphi ^{*}\circ \psi ^{*}$ for isomorphisms~$\varphi \colon \mathcal{G}''\rightarrow \mathcal{G}'$,~$\psi \colon \mathcal{G}'\rightarrow \mathcal{G}$.
\end{rem}
\section{Equivariant Thom classes}
Following Guillemin--Zara~\cite[\S 2.3]{GZequivariant}, we introduce the equivariant Thom class
corresponding to a vertex of a GKM graph~$(\mathcal{G},\alpha)$.

\begin{defi}
  For any~$p\in\mathcal{V}$, we define a map~$\tau_{p}\colon \mathcal{V}\rightarrow H^{*}(BT)$
  by
\[ 
\tau_{p}(q)\coloneqq \begin{cases} \prod_{e\in\mathcal{E}_{p}}\alpha (e) & \text{if $q=p$,} \\ 0 & \text{if $q\not= p$.} \end{cases}
\]
The map~$\tau_{p}$ is called the \textit{equivariant Thom class} associated with~$p$;
it is an element of~$H_{T}^{2n}(\mathcal{G})$.
\end{defi}

\begin{rem}
  \label{rem:masuda-1}
  Assume that $\mathcal{G}$ is the GKM graph of a toric manifold~$X$
  given by a complete fan~$\Sigma$. Then each~$p\in X^{T}=\mathcal{V}$
  is the transverse intersection of $n$~invariant divisors~$X_{i_{1}}$,~\dots,~$X_{i_{n}}$.
  Our~$\tau_{p}$ corresponds
  to the Thom class~$\tau_{i_{1}}\cdots\tau_{i_{n}}\in H_{T}^{2n}(X)$ of~$p$ in~$X$.
  The Thom class~$\tau_{i}$ of~$X_{i}$ (\emph{cf.}~\cite[p.~2007]{Mequivariant})
  is represented by the function~$\xi_{i}\colon\mathcal{V}\to H^{*}(BT)$ given by
  \[
  \xi_{i}(p)=\begin{cases}
    \alpha(e) & \text{if $p\in X_{i}$,} \\
    0 & \text{otherwise}
  \end{cases}
  \]
  where $e\in\mathcal{E}$ is the unique edge from~$p$ to a~$q\notin X_{i}$,
  see~\cite[\S 6.2]{MPcohomology}.
\end{rem}

\begin{lem}\label{characterizationofthom0}
  Let $F$ be a subset of~$H_{T}^{*}(\mathcal{G})\setminus\{0\}$
  such that $fg=0$ for all distinct~$f$,~$g\in F$.
  Then $|F|\le|\mathcal{V}|$. Equality holds if and only if each $f$
  is supported at a single vertex and each
  vertex occurs as the support of some~$f\in F$.
\end{lem}

\begin{proof}
Let $\mathcal{V}_{f}\coloneqq \{ p\in \mathcal{V}\mid f(p)\not= 0\}$ be the support of~$f\in\mathcal{V}$.
  The assumptions imply $\mathcal{V}_{f}\ne\emptyset$ and $\mathcal{V}_{f}\cap\mathcal{V}_{g}=\emptyset$ for
distinct~$f$,~$g$. Thus the inequality holds. 
The equality is attained if and only if
each~$\mathcal{V}_{f}$ is a singleton and $\mathcal{V}$ is the union of the~$\mathcal{V}_{f}$'s.
\end{proof}

The following result gives a ring-theoretic characterization of the set of equivariant Thom classes~$\{ \tau_{p}\}_{p\in\mathcal{V}}$, up to sign.

\begin{pro}
  \label{charequivthom}
  The set~$F=\{\,\tau_{p}\mid p\in\mathcal{V}\,\}$ is a maximal collection of elements as in Lemma~\ref{characterizationofthom0},
  and each other maximal collection is obtained by scaling each equivariant Thom class by some element in~$H^{*}(BT)$.
  These properties characterize $F$ up to signs.
\end{pro}

\begin{proof}
  This follows from Lemma~\ref{characterizationofthom0} and the definition of~$H_{T}^{*}(\mathcal{G})$.
\end{proof}

\section{Key lemma}
Throughout this section, we fix vertices $p\not= q$ of~$\mathcal{G}$. We then introduce the following polynomials:
\[
P\coloneqq \!\! \prod_{e\in \mathcal{E}_{p}\setminus \mathcal{E}_{pq}} \!\! \alpha (e), \qquad Q\coloneqq \!\! \prod_{e\in \mathcal{E}_{q}\setminus \mathcal{E}_{qp}} \!\! \alpha (e).
\]
For any~$e\in \mathcal{E}_{pq}$ we set
\[
c(e)\coloneqq \bigl|\{\, e'\in\mathcal{E}_{pq}\mid e'\not= e, \, \alpha (\overline{e'})=-\alpha (e')\,\}\bigr| .
\]

The following is the key lemma in our proof. For its proof the existence of a parallel transport on~$\mathcal{G}$ is essential.

\begin{lem}\label{keylemma}
  For any~$e\in\mathcal{E}_{pq}$,
  the polynomial $P-(-1)^{c(e)}\,Q$ is divisible by~$\alpha (e)$. 
\end{lem}
\begin{proof}
Let $\mathcal{P}$ be a parallel transport over~$\mathcal{G}$. By condition~\ref{defofparallel-3} in Definition~\ref{defofparallel}, there exist integers~$\{ d_{e,e'}\}_{e'\in\mathcal{E}_{p}}$ satisfying
\[
\alpha (\mathcal{P}_{e}(e'))-\alpha (e')=d_{e,e'}\,\alpha (e)
\]
for any~$e'\in\mathcal{E}_{p}$. Using these relations, we have
\begin{multline*}
  P\cdot\! \prod_{\substack{e'\in \mathcal{E}_{pq} \\ e'\not= e}}\alpha (e') = \prod_{\substack{e'\in \mathcal{E}_{p} \\ e'\not= e}}\alpha (e')
  = \prod_{\substack{e'\in \mathcal{E}_{p} \\ e'\not= e}}\bigl(\alpha (\mathcal{P}_{e}(e'))-d_{e,e'}\,\alpha (e)\bigr)
  \equiv \prod_{\substack{e'\in \mathcal{E}_{p} \\ e'\not= e}}\alpha (\mathcal{P}_{e}(e')) \\
  {}={} \prod_{\substack{e''\in\mathcal{E}_{q} \\ e''\not= \overline{e}}}\alpha (e'')
  {}={} Q\cdot\! \prod_{\substack{e''\in \mathcal{E}_{qp} \\ e''\not= \overline{e}}}\alpha (e'')
  {}={} Q\cdot (-1)^{c(e)}\!\prod_{\substack{e'\in \mathcal{E}_{pq} \\ e'\not= e}}\alpha (e') \; ;
\end{multline*}
here ``$\equiv $" means equality modulo~$\alpha (e)$ (in other words, equality in the quotient ring $H^{*}(BT)/\langle\alpha (e)\rangle$).

Thus 
\[
\bigl(P-(-1)^{c(e)}\,Q\bigr)\!\prod_{\substack{e'\in \mathcal{E}_{pq} \\ e'\not= e}}\alpha (e')
\]
is divisible by~$\alpha (e)$ as $\alpha (e)$ is a prime element in $H^{*}(BT)$. Since $\alpha (e')\ (e'\in\mathcal{E}_{pq}$, $e'\not= e)$ and~$\alpha (e)$ are coprime by the GKM~condition, the proof is now complete.
\end{proof}

We set
\begin{equation*}
  E\coloneqq \{\, e\in \mathcal{E}_{pq}\mid \text{$c(e)$ is even}\,\},
  \qquad
  O\coloneqq \{\, e\in\mathcal{E}_{pq}\mid \text{$c(e)$ is odd}\,\}.  
\end{equation*}
Notice that $\mathcal{E}_{pq}$ is the disjoint union of~$E$ and~$O$. Lemma~\ref{keylemma} immediately implies the following:
\begin{cor}\label{EO}
The polynomials~$P-Q$ and~$P+Q$ are divisible by~$\displaystyle \prod_{e\in E}\alpha (e)$ and~$\displaystyle \prod_{e\in O}\alpha (e)$, respectively.
\end{cor}

\section{Proof of the main theorem}
Now we are in the position to prove our main theorem:
\begin{thm}
  \label{thm:graph-cohom-iso}
$H_{T}^{*}(\mathcal{G})$ and~$H_{T}^{*}(\mathcal{G}')$ are isomorphic as $H^{*}(BT)$-algebras if and only if $\mathcal{G}$ and~$\mathcal{G}'$ are isomorphic as GKM~graphs.
\end{thm}
\begin{proof}
The ``if'' part follows from Remark~\ref{iso}.

 We have seen in Proposition~\ref{charequivthom} that from the $H^{*}(BT)$-algebra~$H_{T}^{*}(\mathcal{G})$
  we can recover the equivariant Thom classes, up to sign, and in particular the vertex set. We show that one can also recover the polynomials~$P_{pq}$.
  This will prove the claim.
  
  Let $p$,~$q$ be distinct vertices. We define a map $f\colon\mathcal{V}\rightarrow H^{*}(BT)$ by
\[
f(v)\coloneqq \begin{cases} P^{2} & \text{if $v=p$,} \\ Q^{2} & \text{if $v=q$,} \\ 0 & \text{otherwise.} \end{cases}
\]
We first check that $f$ is in~$H_{T}^{*}(\mathcal{G})$.

By Corollary~\ref{EO},
$f(p)-f(q)=(P+Q)(P-Q)$ is divisible by
\[
\bigg(\prod_{e\in E}\alpha (e)\bigg)\cdot \bigg(\prod_{e\in O}\alpha (e)\bigg)= P_{pq}.
\]
Together with the definition of~$P$ and~$Q$, this implies that $f$ is an element of~$H_{T}^{*}(\mathcal{G})$.

We note that the identity~$P_{pq}^{2}f=\tau_{p}^{2}+\tau_{q}^{2}$ holds. In the rest of the proof, we show that the degree of~$P_{pq}$ and this identity characterize $\pm P_{pq}$.

Assume that an element~$R\in H^{*}(BT)$ satisfies the identity~$R^{2}g=\tau_{p}^{2}+\tau_{q}^{2}$ for some~$g\in H_{T}^{*}(\mathcal{G})$. Then $g(v)=0$ for all~$v\in\mathcal{V}\setminus \{ p,q\}$, and $g(p)$ is divisible by~$P$.
Since $g(p)=(\tau_{p}(p)/R)^{2}$ and~$P$ is square-free, we see that $g(p)$ is
even divisible by~$P^{2}$.
Because $P_{pq}^{2}P^{2}=P_{pq}^{2}f(p)=R^{2}g(p)$, it follows that $P_{pq}$ is divisible by~$R$. In conclusion, $\pm P_{pq}$ is characterized as such an~$R$ of maximal degree.
\end{proof}

\begin{rem}
  Our proof of Theorem~\ref{thm:graph-cohom-iso}
  is in a sense dual to Masuda's \cite[\S 3]{Mequivariant}.
  In our notation, Masuda starts by essentially
  characterizing the Thom classes~$\xi_{i}$ (see Remark~\ref{rem:masuda-1})
  corresponding to invariant divisors
  as the non-zero elements of~$H_{T}^{2}(\mathcal{G})$ with minimal support.
  He then reconstructs the fan~$\Sigma$ by checking which products among the~$\xi_{i}$'s
  are non-zero.
  In the absence of a fan inducing the graph~$\mathcal{G}$,
  we instead look at elements of~$H_{T}^{*}(\mathcal{G})$ supported
  at single vertices, which correspond to fixed points.
\end{rem}

\begin{ackn} 
The authors are grateful to the referee for several useful suggestions which have improved the presentation of the paper.
M.~F.\ was supported by an NSERC Discovery Grant. H.~Y.\ was supported by JSPS Grant-in-Aid for Early-Career Scientists 19K14537, the bilateral program ``Topology and geometry of torus actions, cohomological rigidity, and hyperbolic manifolds'' between JSPS and RFBR, and Osaka City University Advanced Mathematical Institute (MEXT
Joint Usage/Research Center on Mathematics and Theoretical Physics).
\end{ackn}


\begin{thebibliography}{99}
\bibitem{FPexactcohomology} M.~Franz, V.~Puppe, \textit{Exact cohomology sequences with integral coefficients for torus actions}, Transformation Groups \textbf{12} (2007), 65--76.
\bibitem{FPexactsequence} M.~Franz, V.~Puppe, \textit{Exact sequences for equivariantly formal spaces}, C.\ R.\ Math.\ Acad.\ Sci.\ Soc.\ R.\ Can.\ \textbf{33} (2011), 1--10.
\bibitem{GKMequivariant} M.~Goresky, R.~Kottwitz, R.~MacPherson, \textit{Equivariant cohomology, Koszul duality, and the localization theorem}, Invent.\ Math.\ \textbf{131} (1998), 25--83.
\bibitem{GZequivariant} V.~Guillemin, C.~Zara, \textit{Equivariant de Rham theory and graphs},
  Asian J.\ Math.\ \textbf{3} (1999), 49--76.
\bibitem{MMPtorus} H.~Maeda, M.~Masuda, T.~Panov, \textit{Torus graphs and simplicial posets}, Adv.\ Math.\ \textbf{212} (2007), 458--483.
\bibitem{Mequivariant} M.~Masuda, \textit{Equivariant cohomology distinguishes toric manifolds}, Adv.\ Math.\ \textbf{218} (2008), 2005--2012.
\bibitem{MPcohomology} M.~Masuda, T.~Panov, \textit{On the cohomology of torus manifolds}, Osaka\ J.\ Math.\ \text{43} (2006), 711--746.
\end{thebibliography}
\end{document}